\documentclass{amsart}
\usepackage{amsfonts}
\usepackage{mathrsfs}
\usepackage{amsthm}
\usepackage{txfonts}
\usepackage{palatino,amssymb,epsfig,amsmath,latexsym,amsthm}
\usepackage{palatino,amssymb,epsfig,amsmath}
\usepackage{latexsym ,amsthm,epsf,xypic,epic,amscd}
\usepackage{palatino, amsmath, amsthm, amssymb, amscd, mathrsfs}

\theoremstyle{plain}
\newtheorem{theorem}{Theorem}[section]
\newtheorem{lemma}[theorem]{Lemma}

\newtheorem{proposition}[theorem]{Proposition}
\newtheorem{corollary}[theorem]{Corollary}

\newtheorem{Bounded Diameter Lemma}[theorem]{Bounded Diameter Lemma}
\theoremstyle{definition}
\newtheorem{definition}[theorem]{Definition}

\newtheorem{remark}[theorem]{Remark}

\title{How many cages midscribe an egg}

\author{Jinsong Liu and Ze Zhou}
\address{HUA Loo-Keng Key Laboratory of Mathematics, Chinese Academic of Sciences, Beijing 100190, China }
\address{Institute of Mathematics, Academic of Mathematics $\&$ System Sciences,
Chinese Academic of Sciences, Beijing 100190, China}
\email{liujsong@math.ac.cn zhouze@amss.ac.cn}

\begin{document}

\maketitle

\begin{abstract}
The Midscribability Theorem, which was first proved by O. Schramm,
states that: given a strictly convex body $K\subset\mathbb{R}^{3}$
with smooth boundary and a convex polyhedron $P$, there exists a
polyhedron $Q \subset \mathbb{RP}^3$ combinatorially equivalent to $P$
which midscribes $K$. Here the word "midscribe" means that all it's
edges are tangent to the boundary surface of $K$.

By using of the intersection number technique, together with the
Teichm\"{u}ller theory of packings, this paper provides an alternative
approach to this theorem. Furthermore, combining Schramm's method
with the above ones, the authors prove a
rigidity result concerning this theorem as well. Namely, such a polyhedron is
unique under certain normalization conditions.


\end{abstract}

\setcounter{section}{-1}

\section{Introduction}\label{In}
Let $Q\subset\mathbb{R}^{3}$ be a convex polyhedron, and
$K\subset\mathbb{R}^{3}$ be a strictly convex body with smooth
boundary $\partial K$. We call $Q$ a $K$-midscribable polyhedron if
all its edges (an edge here includes the entire line where the edge
segment belongs to) are tangent to $\partial K$. In addition, when
$K=\mathbb{B}^3$ is the unit ball in $\mathbb{R}^{3}$, we often
call $Q$\, a midscribable polyhedron for short.

It follows from the  Koebe-Andreev-Thurston Theorem (i.e. Circle
Packing Theorem) \cite{And1,And2,Steph,Thu} that: for every given convex polyhedron $P$, there is a convex
midscribable polyhedron $Q$ combinatorially equivalent to $P$. Moreover, the midscribable polyhedron is unique up to M\"obius
transformations which preserve the unit sphere.

\begin{figure}[htbp]\centering
\includegraphics[width=5cm]{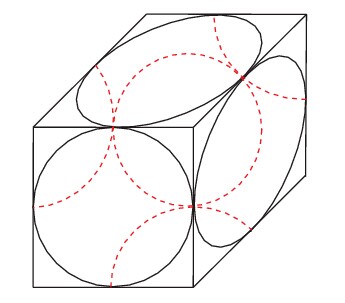}
\caption{A midscribable polyhedron } \label{convex}
\end{figure}

In fact, it's a consequence of the simultaneous realization
phenomenon of circle packings \cite{Boben}. That is, any polyhedral
graph (the $1$-skeleton of a polyhedron) and its dual graph can be
simultaneously realized  by two circle packings such that the two
tangent circles corresponding to an edge in the primal graph and
the two tangent circles corresponding to the dual of this edge are
always orthogonal to each other at the same point in Riemann sphere $\hat{\mathbb C}$.

A generalization of the above result, proposed by E. Schulte, is the
question of replacing the unit ball $\mathbb{B}^3$ by any other
strictly convex body $K \subset \mathbb{R}^3$. Namely, given the
convex body $K \subset \mathbb R^3$, for any convex polyhedron
$P$, is there a $K$-midscribable polyhedron $Q$ combinatorially
equivalent to $P$? In 1992, Schramm proved the following
Midscribability Theorem \cite{Schr2}:

\begin{theorem} \label{Mid} [Midscribability Theorem]
{\it Given a strictly convex body $K\subset\mathbb{R}^{3}$ with
smooth boundary and a convex polyhedron $P$, there exists a convex
$K$-midscribable polyhedron $Q\subset \mathbb R^3$ combinatorially
equivalent to $P$.}
\end{theorem}

Besides, Schramm also proved that the space of all such
$K$-midscribable $Q$ is a $6$-manifold. For instance, when $K$ is the unit ball, Circle Pattern Theorem then implies that this space is identified with the M\"obius group $PSL(2;{\mathbb{C}})$.
Nevertheless, how to characterize this space for general convex bodies seems to
be a difficult problem. In addition, by Circle Pattern Theorem, the rigidity of the
midscribable polyhedron is valid. Analogously, it remains to consider this consequence for general convex bodies. This will be the main purpose of this paper.

In order to attain the target, one method is to introduce proper
normalization conditions. Taking Circle Packing Theorem into consideration, the rigidity property could then be restated as the uniqueness
of circle packings with the central points of three distinct circles are
fixed. We shall treat the Midsribility Theorem by
analogous way.

Given a convex polyhedron $P\subset \mathbb R^3$, we write it as $P\equiv P(\mathcal V,\mathcal E,\mathcal F)$, where $\mathcal V,\mathcal E,\mathcal F$ denote its vertex, edge, and face sets respectively. Choose $f_{0}\in \mathcal F$ and three
ordered sequential edges $e_1,e_2,e_3 \in \mathcal E$ of $f_0$.
We call such 4-tuple $\mathscr O=\{f_0,e_1,e_2,e_3\}$ a
combinatorial frame associated to $P$.

Let $Q\subset \mathbb R^3$ be a $K$-midscribable polyhedron combinatorially equivalent to $P$. The midscribability implies that there exist three tangent points $p_1,p_2,p_3$ corresponding to the edges $e_1,e_2,e_3$ respectively.
Under this convention, we call $Q$ a normalized polyhedron with mark
$\{\mathscr{O},p_1,p_2,p_3\}$.

Then the question becomes: Given a convex polyhedron $P$ with a
combinatorial frame $\mathscr O$, for any three different points
 $p_1,p_2,p_3$ on $\partial K$, is there a convex
$K$-midscribable polyhedron combinatorially equivalent to $P$ with
mark $\{\mathscr{O},p_1,p_2,p_3\}$? In addition, if the answer is
yes, could the uniqueness hold?

Return to the example of the unit ball. From Circle Packing Theorem,
it follows that there exists a convex midscribable polyhedron.
Moreover, all the other midscribable polyhedra can be transformed from
the original one by M\"obius transformations. However, it's possible
that a M\"obius transformation could turn a convex polyhedron into a
non-convex one, or a degenerate one (some points go to infinity).
That means, for some normalized conditions, the corresponding
polyhedron would lose convexity or degenerate.

To overcome these difficulties, we shall use the $3$-dimensional
projective $\mathbb{RP}^3$ instead of the Euclidean space $\mathbb
R^3$. Viewing $\mathbb R^3$ as a subset of
$\mathbb{RP}^3$, we regard the polyhedra in $\mathbb R^3$ as polyhedra in $\mathbb{RP}^3$. Fortunately, there are no pathologies for degenerating polyhedra any more. In summary, we shall prove that:
\begin{theorem}\label{main} Let $K\subset\mathbb{R}^{3}$ be a strictly convex body with
smooth boundary. Given a convex polyhedron $P$ with a combinatorial
frame $\mathscr O$, if $p_1,p_2,p_3$ are three distinct points
on $\partial K$, then there exists a unique $K$-midscribable
polyhedron $Q \subset \mathbb{RP}^3$ with mark
$\{\mathscr{O},p_1,p_2,p_3\}$ which is combinatorially equivalent to
$P$.
\end{theorem}

Denote by $\mathfrak{U}_{P,K}$ the space of such $K$-midscribable
polyhedra in $\mathbb{RP}^3$. Owing to the preceding theorem,
we could identify $\mathfrak{U}_{P,K}$ with the set of distinct points
triple $\{(p_1,p_2,p_3)\}$. It means that $\mathfrak{U}_{P,K}$ is
homeomorphic to the M\"{o}bius group. That is to say
$$
\mathfrak{U}_{P,K}\cong PSL(2;{\mathbb{C}})=Iso^{+}(\mathbb{B}^{3}),
$$
which is a 6-dimensional manifold. Let's consider the subset
$\mathfrak{U}_{P,K}^{c} \subset \mathfrak{U}_{P,K}$, which consists
of all convex $K$-midscribable polyhedra in $\mathbb R^3$. We
have:

\begin{theorem}\label{main1} $\mathfrak{U}_{P,K}^{c}$
is a non-empty open subset of  $\mathfrak{U}_{P,K}$.
\end{theorem}

The proofs are briefly sketched as follows.

Let $K$ be the above convex body. Given an affine half space $H^+$, then the intersection $H^+ \cap \partial K$ is either empty, or a point, or a topological disk. For the last case, we call it a $K$-disk, and its boundary (in $\partial K$) a $K$-circle.

   Recall that $P$ is a convex polyhedron. By associating every face of $P$ with an affine half space, we then obtain a so-called configuration. Denote by $Z$ the space of all such configurations. Then we are interested in two subspaces of $Z$. One is $Z_P$, which represents the space of configurations corresponding to polyhedra which are combinatorially equivalent to $P$. The other is $Z_K$, which consists of configurations corresponding to $K$-circle packings whose contact graph are isomorphic to $G^{\ast}(P)$ ( the dual graph of the $1$-skeleton of $P$).

 In terms of these notions, to find a $K$-midscribable polyhedron means to find a point in the intersection $Z_P\cap Z_K$. In other words, to prove the Midcribility Theorem is equivalent to prove that $Z_P\cap Z_K$ is non-empty.

 By combining the intersection
number theory and a homotopy technique, we will obtain the desired result.

\bigskip
{\bf Notational Conventions}.

Through this paper, for any given set $A$, we use the notation $|A|$ to denote the cardinality of $A$.

\bigskip
\section{Preliminaries}\label{Pn}

Differential topology, especially transversality and intersection
number theory, will play an important part in this paper. In this
section, let's give a simple introduction to them. The reader is
referred to \cite{Gui, Hir} for a detailed exposition of the general
theory of differential topology.

The first topic is transversality, according to H.E.Winkelnkemper,
which is said to unlock the secrets of the manifolds
(see Chap.3 in \cite{Hir}). Indeed, it plays a significant role throughout the paper \cite{Schr2}.

Suppose $M,N$ are two oriented smooth manifolds, and $S\subset N$ is a submanifold.

\begin{definition}
Assume that $f:M\rightarrow N$ is a ${C}^{1}$ map. Given $A\subset M$, we say $f$ is transverse to $S$ along $A$,
denoted by $f\mbox{\Large $\pitchfork$}_{A}S$, if
$$
Im(df_{x})+T_{f(x)}S=T_{f(x)}N
$$
whenever $x\in A\cap f^{-1}(S)$. When $A=M$, we simply denote
$f\mbox{\Large $\pitchfork$} S$.
\end{definition}

The other notion is the intersection number.
Let $S\subset N$ be a closed submanifold such that
$$
dim M + dim S = dim N.
$$
 Suppose $\Lambda \subset M$ is an open subset with compact closure $\bar{\Lambda}$. Given a continuous map $f:M\rightarrow N$ such that $f(\partial \Lambda)\cap S=\emptyset$, where $\partial\Lambda=\bar{\Lambda}\setminus\Lambda$, we will define a topological invariant $I(f,\Lambda,S)$, called the intersection number between $f$ and $S$ in $\Lambda$.

 If $f\in C^{0}(\bar{\Lambda},N)\cap C^{\infty}(\Lambda,N)$ such that $ f\mbox{\Large $\pitchfork$}_{\Lambda}S$,
then $\Lambda \cap f^{-1}(S)$ consists of finite points. For each $x
\in \Lambda \cap f^{-1}(S)$, the $sgn(f,S)_{x}$
at $x$ is $+1$, if the orientations on $Im(df_{x_{j}})$ and
$T_{f(x_j)}S$ "add up" to preserve the prescribed orientation on
$N$, and $-1$ if not.

\begin{definition}
If $\Lambda \cap f^{-1}(S)=\{x_1,x_2,\cdots,x_m\}$, then we define
the intersection number between $f$ and $S$ in $\Lambda$ to be
$$
I(f,\Lambda,S):=\sum\nolimits_{j=1}^m sgn(f,S)_{x_j}.
$$
\end{definition}

The proof of the following proposition is in the same style as that
of the homotopy invariance of Brouwer degree. Please see Milnor's
book \cite{Mil}.

\begin{proposition}\label{Prop1}
Suppose that $f_{i}\in C^{0}(\bar{\Lambda},N)\cap
C^{\infty}(\Lambda,N)$, $f_{i}\mbox{\Large $\pitchfork$}_{\Lambda}S
$ and $f_{i}(\partial \Lambda)\cap S=\emptyset, \:i=0,1$. If there
exists a homotopy
$$
H \in C^{0}({I\times \bar\Lambda},N)
$$
such that $H(0, \cdot)=f_0(\cdot), \:\: H(1, \cdot)=f_1(\cdot)$, and
 $H(I\times\partial \Lambda)\cap S=\emptyset$, then
$$
I(f_{0},\Lambda,S)=I(f_{1},\Lambda,S).
$$
\end{proposition}

The next lemma, which helps us to manipulate the intersection number for
general mappings, is a consequence of Sard's theorem \cite{Gui,Hir}.
\begin{lemma}\label{Lem1}
For any $f\in C^{0}(\bar{\Lambda},N)$ with $f(\partial \Lambda)\cap
S=\emptyset$, there exists $g\in C^{0}(\bar{\Lambda},N)\cap
C^{\infty}(\Lambda,N)$ and $H\in C^{0}(I\times\bar{\Lambda},N)$ such
that
\begin{itemize}
\item[$(1)$] $g\mbox{\Large $\pitchfork$}_{\Lambda}S$;
\item[$(2)$] $H(0,\cdot)=f(\cdot),H(1,\cdot)=g(\cdot)$;
\item[$(3)$] $H(I\times\partial \Lambda)\cap S=\emptyset$.
\end{itemize}
\end{lemma}
The above lemma, together with Proposition \ref{Prop1}, allows one
to define the intersection numbers for general continuous mappings.

\begin{definition}
For any $f\in C^{0}(\bar{\Lambda},N)$ with $f(\partial \Lambda)\cap
S=\emptyset$, we can define the intersection number
$$
I(f,\Lambda,S)=I(g,\Lambda,S).
$$
\end{definition}
By Proposition \ref{Prop1}, $I(f,\Lambda,S)$ is well-defined.
Furthermore, we have the following homotopy invariance property of
this quantity.
\begin{theorem}\label{Thm1}
For $i=0,1$, suppose that $f_{i}\in C^{0}(\bar{\Lambda},N)$ such
that $f_{i}(\partial \Lambda)\cap S=\emptyset$. If there exists
$H\in C^{0}(I\times\bar{\Lambda},N)$ such that
\begin{itemize}
\item[$(1)$] $H(0,\cdot)=f_{0}(\cdot),H(1,\cdot)=f_{1}(\cdot)$;
\item[$(2)$] $H(I\times\partial \Lambda)\cap S=\emptyset$.
\end{itemize}
Then we have $I(f_{0},\Lambda,S)=I(f_{1},\Lambda,S)$.
\end{theorem}

In particular, it immediately follows from the definition that:
\begin{theorem}\label{Thm3}
If $I(f,\Lambda,S)\neq 0$, then we have $\Lambda \cap
f^{-1}(S)\neq\emptyset$.
\end{theorem}

\bigskip
\section{Teichm\"{u}ller theory of packings}\label{Cp}
 Roughly speaking, a packing $\mathcal {P}$ is a configuration
of topological disks(circles) $\{D_{v}:v\in V\}$  with specified patterns of tangency. The
contact graph (or nerve) of $\mathcal{P}$ is a graph
$G_{\mathcal{P}}$, whose vertex set is $V$ and an edge appear if and
only if the corresponding disks(circles) touch.

Recall that a $K$-disk is defined as the intersection $H^+ \cap \partial K$, where $H^+$ is an affine half space which intersects the interior of the convex body $K$. Naturally, we call $\mathcal {P}=\{D_{v}:v\in V\}$ a $K$-circle packing, if all $D_{v}(v\in V)$ are $K$-disks.

In this section, we shall
investigate the  Teichm\"{u}ller theory of such packings, which characterizes $Z_K$, the space of $K$-circle packings with the same contact graph. To reach the purpose, a
main step is to generalizes the Circle Packing Theorem.

Given the convex body $K\subset \mathbb{R}^3$, without loss of generality, we assume it lies below the
plane $\{(x,y,z)\in \mathbb R^3: z=1\}$ and is tangent to the plane
at the point $N=(0,0,1)$, which is regarded as the "North Pole"
of $\partial K$. Let $h:
\partial K \rightarrow \mathbb C \cup \{\infty\}$ denote the "stereographic
projections" with $h(N)=\infty$. Since $h$ can be extended to be a
diffeomorphism between $\partial K$ and $\hat{\mathbb C}$, we endow $\partial K$ with a complex structure by pulling back the complex structure of $\hat{\mathbb C}$. Hence, up to conformal equivalence,  $\partial K$ identifies with the Riemann sphere $\hat{\mathbb C}$. Therefore, it's plausible to introduce the following notion:

\begin{definition}
A $K$-circle domain in the Riemann sphere $\hat{\mathbb C}$ ($\partial K$) is a
domain, whose complement's connected components are all closed
$K$-disks and points.
\end{definition}

Let $\Omega_{n} \subset \hat{\mathbb{C}}$ be a finitely connected
domain with $n$ boundary components. A marked domain
$\Omega_{n}(z_1,z_2,z_3)$ is the domain $\Omega_{n}$ together with
three different ordered points $z_1,z_2,z_3$ in the same boundary
component. In \cite{Schr3}, Schramm  proved the following result,
which generalizes Koebe's Uniformization \cite{H-S} of finitely connected domains:

\begin{lemma}\label{Koebe}
Let $\Omega_{n}(z_1,z_2,z_3)$ be a marked $n$-connected domain in
$\hat{\mathbb{C}}$. For any given points triple $\{p_1,p_2,p_3\},
p_i \in \partial K$, there exist a marked $K$-circle domain
$\Omega_{n}^{K}(p_1,p_2,p_3)$ and a conformal mapping
$f:\Omega_{n}(z_1,z_2,z_3)\rightarrow \Omega_{n}^{K}(p_1,p_2,p_3)$
such that $f(z_i)=p_i, \: i=1,2,3$.
\end{lemma}

 Given a polyhedral graph $G(V,E)$, let's choose  a vertex $v_{0}\in V$ and three ordered edges $e_1,e_2,e_3 \in E$ emanating from $v_0$.
Similarly, we call the 4-tuple $\mathscr{O}=\{v_0,e_1,e_2,e_3\}$
a combinatorial frame associated to the graph $G$. Suppose
$\mathcal {P}=\{D_{v}\}$ is a packing with the contact graph $G_{\mathcal {P}}=G(V,E)$. Denoting
by $p_1,p_2,p_3$ the three tangent points corresponding to the edges
$e_1,e_2,e_3$, we call $\mathcal {P}$ a normalized packing with mark
$\{\mathscr{O},p_1,p_2,p_3\}$. Under these conventions,
turning to the limiting case of Lemma \ref{Koebe}, we have:

\begin{corollary}\label{Cor}
  Let $G(V,E)$ be a polyhedral graph associated with a combinatorial frame $\mathscr{O}$. If $\{p_1,p_2,p_3\}$ is any
three points triple in $\partial K$, then there exists
a $K$-circle packing $\mathcal {P}_{K}=\{D_v:v\in V\}$ realizing $G(V,E)$ as its contact
graph. Moreover, $\mathcal {P}_{K}$ is normalized with mark $\{\mathscr{O},p_1,p_2,p_3\}$.
\end{corollary}
If $G(V,E)$ is a triangular graph, Schramm \cite{Schr1} proved that
such a packing is unique. While $G(V,E)$ isn't a triangular graph, the
uniqueness wouldn't hold any more. In fact, there exist uncountable
normalized packings with the same contact graph. To
characterize this problem, He-Liu \cite{H-L} developed the
Teichm\"{u}ller theory of circle patterns (packings). Recalling their method, it seems little hard to consider similar results of the $K$-circle packings.

\bigskip
We would employ the notions of conformal polygons, which are considered as analogs of the conformal quadrangles. In fact, it's defined as pairs $h:I\rightarrow \hat{\mathbb{C}}$, where $I\subset\hat{\mathbb{C}}$ is a given topological polygon and $h$ is a quasiconformal embedding. For details on quasiconformal mappings,  please refer to Ahlfors' book \cite{Ahl}.

 Say two such quasiconformal embedding $h_1,h_2: I\rightarrow
\hat{\mathbb{C}}$ are Teichm\"{u}ller equivalent, if the composition
mapping $h_2\circ (h_{1})^{-1}: h_1(I)\rightarrow h_2(I)$ is
isotopic to a conformal homeomorphism $f$ such that for each
side $e_{i}\subset\partial I$, $f$ maps
$h_1(e_{i})$ onto $h_2(e_{i})$.

\begin{definition}
The Teichm\"{u}ller space of $I$, denoted by $\mathcal{T}_I$, is the
space of all equivalence classes of quasiconformal embeddings
$h:I\rightarrow \hat{\mathbb{C}}$.
\end{definition}
\begin{remark}\label{Rem}
If the polygon $I$ is  $k-$sided, it follows from the classical
Teichm\"{u}ller theory that $\mathcal{T}_I$ is diffeomorphic to the
Euclidean space $\mathbb{R}^{k-3}$. See e.g \cite{LV}.
\end{remark}

Recall that $G^*(P)=(V, E)$ is the 1-skeleton of the dual polyhedron of $P$, where $P$ has been written as  $P\equiv P(\mathcal V,\mathcal E,\mathcal F)$. Let us fix a circle packing $\mathcal{P}_0=\{D_{0,v}\}$ on the
unit sphere $\mathbb S^2(=\partial \mathbb B^3\cong\hat{\mathbb C})$ with the contact graph $G^*(P)$. For any component $I_i$ of
$\hat{\mathbb{C}}-\cup_{v\in V}D_{0,v}$, we call it an interstice. Evidently, $I_i$ is a topological polygon. Thus we could associate it with the Teichm\"{u}ller space $\mathcal{T}_{I_i}$.
Denote $\mathcal {T}_{G^*(P)}=\prod_{i=1}^{m}\mathcal {T}_{I_i}$,
where $\{I_1,I_2,\cdots, I_m\}$ are all interstices of the circle
packing $\mathcal P_{0}$. Then $m=|F|$, where $F$ is the face set of the
contact graph $G^*(P)$. Due to Remark \ref{Rem}, we easily
check that $\mathcal {T}_{G^*(P)}\cong \mathbb{R}^{2|E|-3|F|}=\mathbb{R}^{2|\mathcal E|-3|\mathcal V|}$.
Analogous to He-Liu \cite{H-L}, we now establish the following
theorem.

\begin{theorem}\label{Deform}
Let $K$, $P$, $G^*(P)$ and $\mathcal {T}_{G^*(P)}$ be as above. Suppose $p_1, p_2, p_3$ are three different points in $\partial K$.
For any
 \[
 \big[\tau\big]=([\tau_{1}],[\tau_{2}],\cdots,[\tau_{m}]) \in \mathcal {T}_{G^*(P)},
 \]
there exists a unique $K$-circle packing $\mathcal
{P}_{K}([\tau])$ with mark $\{\mathscr{O},p_1,p_2,p_3\}$
realizing $G$. Moreover, the interstice of $\mathcal
{P}_{K}([\tau])$ corresponding to $I_i$ is endowed with the given
complex structure $[\tau_{i}],\:1\leq i\leq m$.
\end{theorem}

\begin{proof}
We will prove the theorem by two steps.

\textbf{Existence.}
In fact, as a limiting case of Lemma \ref{Koebe}, it's straightforward. However, we would take another approach, which seems to provide more intuitive understanding into this consequence.
The ideal is similar to the method used by He-Liu in \cite{H-L} or
Schramm in \cite{Schr4}. It is a combination of the general Packing Theorem
\cite{Schr1,Schr3} and Rodin-Sullivan method \cite{R-S}.

For each $I_i \in I=\{I_1,I_2,\cdots, I_n\}$, and for any given complex
structure $[\tau_i]:I_i \rightarrow\hat{\mathbb{C}}$, without loss of
generality, we may assume that the image region $[\tau_i](I_i)$ is a
bounded domain in the complex plane $\mathbb{C}$.

Lay down a regular hexagonal packing of circles in $\mathbb{C}$, say each of radius
$1/n$. By using the boundary component $\partial[\tau_i](I_i)$ like a
cookie-cutter, we obtain a circle packing $\mathcal {Q}_{I_i,n}$ which
consists of all the circles intersecting the region $[\tau_i](I_i)$.
Denote by $\mathcal {G}_{I_i,n}$ the contact graph of $\mathcal {Q}_{I_i,n}$.

Joining the contact graphs $\{\mathcal {G}_{I_i, n}\}_{1\leq i\leq
m}$ to the original contact graph $G^*(P)$ along the corresponding
edges, we obtain a triangular graph $G_n$. For every newly obtained vertex of $G_n$, we associate it with the standard disk foliation. According to the general Packing Theorem \cite{Schr2,Schr4}, there exists a normalized  packing $\mathcal P_{n}$
realizing the contact graph $G_{n}$. Discarding the standard disks corresponding to those new vertexes, we acquire a packing realizing $G^*(P)$. For ease of notations, we still denote it by $\mathcal P_{n}$.

For packing sequence $\mathcal P_{n}$, it's easy to see that there
exists subsequence $\mathcal P_{n_k}$ such that $\mathcal P_{n_k}$ convergent to a
pre-packing $\mathcal P_{\infty}$. Because of the following Proposition
\ref{Prop2}, this pre-packing isn't degenerate. That means $\mathcal P_{\infty}$ is
actually a "real" $K$-circle packing with contact graph $G^*(P)$.

Let $\mathcal {P}_{K}([\tau])=\mathcal {P}_{\infty}$. Using Rodin-Sullivan's method \cite{R-S}, it's not hard to verify
that the interstices of $\mathcal {P}_{K}([\tau])$ are endowed with the
given complex structure.

\bigskip
\textbf{Uniqueness}. Suppose that there're
two normalized packings $\mathcal{P}_{K}=\{D_{v}\}$ and
$\mathcal{P'}_{K}=\{D'_{v}\}$ satisfying the same conditions.
Namely, there exists conformal map $f:I\rightarrow I'$ between the
corresponding interstices of $\mathcal{P}_{K}$ and
$\mathcal{P'}_{K}$, we shall prove the uniqueness by the Rigidity Lemma (Lemma \ref{Lem7}) in Appendix part.
It completes the proof.
\end{proof}

\begin{proposition}\label{Prop2}
The packing $\mathcal {P}_{\infty}=\{D_{\infty,v}:v\in V\}$ isn't degenerating.
\end{proposition}
\begin{proof}
Suppose it's not true. Then there exists at least one disks tends to a point. Note that any three $K$-disks
with disjoint interiors can not meet at a common point. Therefore,
 all $K$-disks in the
packing $\mathcal P_{n_k}$, except for at most two, will degenerate to points, which contradicts to our
normalization conditions.
\end{proof}

\section{Proof of the main theorems}\label{Aap}
The real $3$-dimensional projective space $\mathbb{RP}^3$ is the
space of all lines through $0$ in $\mathbb R^4$. More precisely, we
can define $x\thicksim x'$ in $\mathbb R^4\backslash \{0\}$ if and
only if there is a real number $\lambda\neq 0$ such that $x'=\lambda
x$. Let
$$
\Pi': \mathbb R^4\backslash \{0\}\rightarrow \mathbb{RP}^3
$$
be the projection. Then we denote the point $\Pi'((x_0,
x_1,x_2,x_3))$ by $[x_0, x_1,x_2,x_3]$, which is the homogeneous
coordinates in $\mathbb{RP}^3$. We also have the projection
$$
\Pi: \mathbb S^3 \rightarrow \mathbb{RP}^3,
$$
where $\mathbb S^3$ is the unit sphere in the space $\mathbb
R^4\backslash \{0\}$. Then $\mathbb{RP}^3$ can be regarded as the
quotient of $\mathbb S^3$ obtained by identifying antipodal points.

Furthermore, we can regard $\mathbb R^3$ as a subset of
$\mathbb{RP}^3$ by identify the point $[1, x_1,x_2,x_3] \in
\mathbb{RP}^3$ with $( x_1,x_2,x_3) \in \mathbb R^3$. Namely,
$\mathbb{RP}^3=\mathbb R^3 \cup \mathbb{RP}^2$.

Each plane $f \subset \mathbb{RP}^3$ can be defined as
$$
f=\{(x,y,z,w): A x+B y+C z+D w=0\}.
$$
Therefore, each plane $f \subset \mathbb{RP}^3$ is uniquely
determined by the point $[A,B,C,D] \in \mathbb{RP}^3$.


\bigskip
 Recall that $P(\mathcal V,\mathcal E,\mathcal F)\subset \mathbb{R}^{3}$ is a given polyhedron. Let $Z$ denote the space
$(\mathbb{RP}^3)^{|\mathcal F|}$. Namely, a point $z\in Z$ gives a
choice of a half space(or an oriented plane) for each $\mbox{\small
$f$}\in \mathcal F$. $Z$ will be called the configuration space, and
a point $z\in Z$ will be called a configuration. For a configuration
$z\in Z$, we denote by $z_f$ the oriented plane corresponding to the
face $\mbox{\small $f$}\in \mathcal F$.

For any $v\in \mathcal V$, let $lk(v)$ be the number of faces
linking to this vertex $v$. Denote by $\{f_1,f_2,\cdots,
f_{lk(v)}\}$ all faces of $P$ link to the vertex $v$. Let
$Z_{vo}\subset Z$ be the set of configurations $z$ such that, for at
least one triple $\{i_1,i_2,i_3\} \subset \{1,2,\cdots, lk(v)\}$,
the intersection
$$
z_{f_{i_1}}\cap z_{f_{i_2}} \cap z_{f_{i_3}}
$$
contains more than one points. Evidently, $Z_{vo}\subset Z$ is
closed which implies that
\[
Z_{oc}=Z\setminus(\cup_{v\in \mathcal V}Z_{vo})
\]
is open in $Z$. Namely, it's a manifold with the same dimension as
$Z$. More precisely,
\[
dim Z_{oc}=3|\mathcal F|.
\]
On the other hand, let $Z_{vc}\subset Z_{oc}$ be the set of
configurations $z$ such that  $ \cap_{i=1}^{lk(v)}{z_{f_i}}\neq
\emptyset$, where $\mbox{\small $f_1,f_2,\cdots,f_{lk(v)}$}$ are all
faces of $P$ linking to the vertex $v$. We then define
$Z_P=\cap_{v\in \mathcal V}Z_{vc}$. Obviously, a configuration of
$Z_P$ corresponds to a polyhedron in $\mathbb{RP}^3$ combinatorially
equivalent to $P$. We have:

\begin{lemma}\label{Lem3}
$Z_P$ is a closed submanifold of $Z_{oc}$ with real dimension $dim
Z_{P}=|\mathcal E|+6$.
\end{lemma}

\begin{proof}
 Denote by
$f_1,f_2,\cdots,f_{lk(v)}$ the oriented faces of the polyhedron $P$
link to $v$. For every $i=1,2,\cdots, lk(v)$, suppose that the plane
$z_{f_{i}}$ is defined by the equations
\[
z_{f_{i}}=\{[x,y,z,w]: A_{i}x+B_{i}y+C_{i}z+D_{i}w=0\}.
\]

Consider the matrix
$$
\begin{pmatrix}
  A_1&B_1&C_1&D_1\\
  A_2&B_2&C_2&D_2\\
  A_3&B_3&C_3&D_4\\
  \vdots&\vdots&\vdots&\vdots\\
  A_{lk(v)}&B_{lk(v)}&C_{lk(v)}&D_{lk(v)}
\end{pmatrix}
$$
Then $\cap_{i=1}^{lk(v)} z_{f_i}\neq \emptyset$ if and only if the
rank of the above matrix is less than $4$. Equivalently, the
determinant
 $$
  R(f_{i_1},f_{i_2},f_{i_3},f_{i_4})=
  \left|\begin{array}{cccc}
  A_{i_1}&B_{i_1}&C_{i_1}&D_{i_1}\\
  A_{i_2}&B_{i_2}&C_{i_2}&D_{i_2}\\
  A_{i_3}&B_{i_3}&C_{i_3}&D_{i_3}\\
  A_{i_4}&B_{i_4}&C_{i_4}&D_{i_4}
  \end{array}\right|=0.
  $$
for each subset $\{i_1,i_2,i_3,i_4\} \subset \{1,2,\cdots, lk(v)\}$.
From the definition of $Z_{oc}$, $0$ is a regular value of the
smooth function $R$. It follows from the regular value theorem that
$Z_P$ is a closed submanifold of $Z_{oc}$. Please refer to \cite{Hir}.

Then we have:
$$
dim Z_{P}=3|\mathcal F|-(
   \sum\nolimits_{v\in \mathcal V}
  lk(v)-3)=3|\mathcal F|-(2|\mathcal E|-3|\mathcal V|)=|\mathcal E|+6,
$$
where the last identity comes from Euler's formula.
\end{proof}

\bigskip
 Let $G^{\ast}(P)=(V,E,F)$ be as above. We choose a combinatorial frame $\mathscr{O}$ for
$G^{\ast}(P)$ and three different points $p_1,p_2,p_3$ in $\partial
K$. For each $[\tau]\in \mathcal {T}_{G^{\ast}(P)}$, from Theorem
\ref{Deform}, it follows that there is a unique normalized
$K$-circle packing $\mathcal{P}_{K}([\tau])$ with mark
$\{\mathscr{O},p_1,p_2,p_3\}$, which realizes the graph
$G^{\ast}(P)$. Consequently, it gives rise to the following
mapping:
\[
f_{K}:\mathcal {T}_{G^{\ast}(P)}\longrightarrow Z_{oc}\hookrightarrow Z.
\]

In addition, a simple computation shows that:
$$
\begin{aligned}
&dim Z_{P}=3|\mathcal F|-(2|\mathcal E|-3|\mathcal V|)=|\mathcal E|+6,\\
&dim \mathcal {T}_{G^{\ast}(P)}=2|\mathcal E|-3|\mathcal V|\\
&dim \mathcal {T}_{G^{\ast}(P)}+dim Z_{P}=3|\mathcal F|=dim Z_{oc}.\\
\end{aligned}
$$

These identities remind us of the intersection number theory.
Whereas, in order to apply this tool, it's necessary to find a
compact set $\Lambda\subset \mathcal{T}_{G^{\ast}(P)}$, and determine the
intersection number $I(f_{K},\Lambda,Z_{P})$. Indeed, the following
lemma guarantees the existence of such a subset.

\begin{lemma}\label{Lem2}
For any given strictly convex body $K$, there exists a compact set
$\Lambda \subset \mathcal {T}_{G^{\ast}(P)}$ such that
$f(\partial \Lambda)\cap Z_P=\emptyset$.
\end{lemma}

\begin{proof}
We assume, by contradiction, that there is not such a compact set
$\Lambda$. Then there is a sequence of $[\tau]_{n}\in
f_{K}^{-1}(Z_P)$ such that the corresponding normalized packings
$\mathcal {P}_{n}$ satisfy one of the follow two possibilities:

\begin{itemize}
\item{}As $n\rightarrow \infty$, there exists $v\in V$, such that the corresponding circles
$\{{D}_{ n}(v)\}$ in the packings  $\mathcal{P}_{K}([\tau]_n)$ tends to a
point;
\item{} For certain $f\in F$, as $n\rightarrow \infty$, the distance of two non-adjacent arcs of the interstice $I_{f,n}$
of the packings $\mathcal{P}_{K}([\tau]_n)$  tends to zero.
\end{itemize}

Using a similar argument as in Proposition
\ref{Prop2}, we  rule out the first possibility.

In the second case, for given $n$, since $[\tau]_{n}\in f_{K}^{-1}(Z_P)$, this corresponds to a $K$-midscribable polyhedron $P_{n}$. Hence the tangent edges of $P_{n}$
will separate the non-adjacent arcs. On the other hand, we have known that the sizes of all disks in
$\mathcal{P}_{K}([\tau]_{n})$ have positive infimum. These facts
together tell us that the distance of such non-adjacent arcs can't
tend to zero, which prove the statements.
\end{proof}

If we could prove $I(f_{K},\Lambda,Z_{P})\neq 0$, then Theorem
\ref{Thm3} implies that $f^{-1}_{K}(Z_P)\cap \Lambda \neq
\emptyset$, which leads to the existence part of Theorem \ref{main}.
To obtain the desired result, we need the following transversality
theorem, which is a tinily modified version of Schramm's result in
\cite{Schr2}.

\begin{lemma}\label{Lem6}\textbf{(Transversality theorem)}
Given any strictly convex body $K\subset \mathbb{R}^{3}$ with smooth
bounary, then we have $f_{K}\mbox{\Large$\pitchfork$}Z_{P}$.
\end{lemma}

\begin{remark}
It's worth pointing out that, the distinction between Schramm's
theorem and ours, lies in the description of configuration spaces, which is far from
essential. In other words, the above lemma could be deduced by analogous way, as long as we do little modifications to adapt for our definitions.
\end{remark}

With the help of the preceding results, we shall compute the
intersection number by means of homotopy method.

Note that $\mathbb{B}^3\subset \mathbb{R}^3$ is the unit ball and
$K$ is the given convex body. Without loss of generality, we assume
that its diameter is greater than $1$ and $\mathbb{B}^3 \subset K$
with the boundary $\partial K$ tangent the unit sphere
$\mathbb{S}^{2}=\partial{\mathbb{B}^3}$ at the point $N=(0,0,1)$.
Then $N=(0,0,1)$ could be considered as the common "North Pole" of
$\mathbb{B}^3$ and $K$.

Let $h_0$, $h_1$ be the "stereographic projections" for
$\mathbb{S}^{2}=\partial \mathbb{B}^3$, $\partial K$, respectively.
Define a one parameter family of closed surfaces by
\[
\{s \cdot h_1^{-1}(z)+(1-s) \cdot h_0^{-1}(z): z \in \hat{\mathbb
C}\}.
\]
For each $s \in [0,1]$, the above set is a compact strictly convex
surface in $\mathbb R^3$. Denote by $K_s$ the convex body bounded by
this surface. Then $\{K_s\}_{1\leq s \leq 1}$ is a family of
strictly convex bodies joining $\mathbb{B}^3$ and $K$. Each of the
convex body $K_s$ is tangent to the plane $\{(x,y,z)\in \mathbb R^3:
z=1\}$ at $N=(0,0,1)$ from the same side of $K_0$. By the
"stereographic projection", we can identify $\partial K_s$ with
$\hat{\mathbb{C}}$ for each $s \in [0,1]$. Moreover, the curve $s
\rightarrow K_s$ is continuous in the Hausdorff metric.

For each $K_s$, in view of Theorem \ref{Deform}, we can construct a
mapping
 \[
 f_{s}=f_{K_s}:\mathcal{T}_{G^{\ast}(P)}\rightarrow Z_{oc}.
 \]
Moreover, owing to Lemma \ref{Lem2}, there exists $\Lambda
\subset \mathcal{T}_{G^{\ast}(P)}$ such that $f_s(\partial \Lambda)\cap
Z_P=\emptyset$ for all $s \in [0,1] $. Note that $K_0=\mathbb B^3$ and
$K_1=K$. Since $f_{s}$ is a homotopy from $f_{K_0}$ to $f_{K}$, we
conclude that:

\begin{theorem}\label{Thm2}
Given $P$, $K$, $\Lambda$, and $f_{K}$ as above, then $I(f_{K},\Lambda,Z_{P})=1$
\end{theorem}
\begin{proof}
Due to Theorem \ref{Thm1}, we need only to calculate
$I(f_{K_0},\Lambda,Z_{P})$. From Circle Pattern Theorem
\cite{And1,And2,Steph,Thu}, it follows that there is only one point in
$f_{K_0}^{-1}(Z_P)\cap \Lambda$. On the other hand, the above
Transversality theorem tells us that $f_{K_0}$ is transverse to
$Z_P$ at the intersection point, which implies
$I(f_{K_0},\Lambda,Z_{P})=1$. It proves $I(f_{K},\Lambda,Z_{P})=1$.
\end{proof}

\begin{remark}
The Transversality theorem (Lemma \ref{Lem6}) is powerful, but its proof \cite{Schr2} is a little technically involved. On account of this fact, in next section,
we will seek another approach to Theorem \ref{Thm2}, which is independent from Lemma \ref{Lem6}.
\end{remark}

\bigskip
Up to now, we have accomplished the necessary results for our
purpose. It's ready to prove the main consequences of this paper.

\begin{proof}[\textbf{Proof of Theorem \ref{main}}]
The existence part is an immediate result of Theorem \ref{Thm3}
and Theorem \ref{Thm2}. From $f_s\mbox{\Large$\pitchfork$}Z_{P}$, it
follows that the cardinality $|f_{s}^{-1}(Z_P)\cap \Lambda|$ is
constant for $0\leq s\leq 1$. Since $|f_{K_0}^{-1}(Z_P)\cap
\Lambda|=1$, we prove the uniqueness part of the theorem.
\end{proof}

\begin{proof}[\textbf{Proof of Theorem \ref{main1}}]
The openness is a direct corollary of the Transversality theorem.
 For non-emptiness, it's proposal to consider the set $\partial \mathfrak{U}^{c}_{P,K}$, which is the boundary of $\mathfrak{U}^{c}_{P,K}$ in $\mathfrak{U}_{P,K}$. If we could show that $\partial \mathfrak{U}^{c}_{P,K}$ is non-empty, then $\mathfrak{U}^{c}_{P,K}$ must be non-empty.

 Given any $x \in \mathbb{RP}^3
\backslash K$, let $O_x$ be the set of points on $\partial K$ which are visible from $x$. That is, the set of $p\in \partial K$ such that the ray $\overrightarrow{xp}$ and $\vec{n}_p$ form an angle $\theta_p \in[0,\pi/2]$, where $\vec{n}_p$ is the inner normal vector of the smooth surface $\partial K$ at $p$. Clearly, $O_x$ is a topological disk.

We have proved that, for any points triple $(p_1,p_2,p_3)$, there is
a polyhedron $Q=Q(p_1,p_2,p_3) \subset \mathbb{RP}^3$
combinatorial equivalent to $P$, which midscribes $K$. For every $v\in \mathcal V$, let $x(v)$ be the apex of $Q$ corresponding to $v$. Then
 $\{O_{x(v)}\}_{v \in \mathcal V}$ forms a packing on $\partial K$ with the contact graph $G(P)$, where $G(P)$ is the 1-skeleton of $P$.

Choose some $(p_1,p_2,p_3)$ such that the vertex $x(v_0)$ of
$Q=Q(p_1,p_2,p_3)$ is at $\mathbb{RP}^3 \backslash \mathbb R^3$.
Then we have
$$
O_{x(v)} \varsubsetneqq \partial K\setminus O_{x(v_0)}, \:\: v \neq
v_0,
$$
which implies that the remaining vertices $x(v) \in \mathbb R^3$, \: for every $v_0\neq v \in \mathcal V$.

That means the configuration of $Q=Q(p_1,p_2,p_3)$ locates in $\partial \mathfrak{U}^{c}_{P,K}$. We thus complete the proof of Theorem \ref{main1}.

\end{proof}

\bigskip
\section{Further discussion}\label{PT}

Consider the Klein model of the closed  unit ball $\mathbb{B}^ {3}$
in $\mathbb{R}\mathbb{P}^{3}$. A hyperideal polyhedron $P_{hi}$ is
defined to be a compact convex polyhedron in $
\mathbb{R}\mathbb{P}^{3}$ whose vertices locate outside of the
closed unit ball $\mathbb B^3$ and whose edges all meet $\mathbb
B^3$. In 2002, Bao-Bonahon \cite{Bao}
classified the hyperideal polyhedron, up to isometries of $\mathbb{B}^ {3}$, in terms of the combinatorial type and dihedral angles.

Recall that $G^{\ast}(P)=(V,E,F)$ is an embedded graph in
$\mathbb S^2=\partial \mathbb B^3$. They have:

\begin{lemma}\label{Lem4}
Let $\theta_{e}\in(0,\pi]$ be a weigh attached to each edge of $e
\in E$ with the following conditions:
\begin{itemize}
\item[$(i)$] If a simple closed curve formed by edges $e_0,e_1,\cdots,e_n$
of $E$, then $\sum\nolimits_{i=1}^{n}\theta_{e_i}> 2\pi$;

\item[$(ii)$] If a simple arc $\gamma$ of $G^{\ast}(P)$ formed by edges
$e_0,e_1,\cdots,e_n$ joining two distinct vertices $v_1, v_2$ which
are in the closure of the same component $A$ of
$\mathbb{S}^{2}-G^{\ast}(P)$, and if $\gamma$ is not contained in the
boundary of $A$, then$\sum\nolimits_{i=1}^{n}\theta_{e_i}>\pi$.
\end{itemize}
Then there exists a hyperideal polyhedron $P_{hi}$ combinatorially
to $P$ and with external dihedral angle given by $\theta_{e}$.
Moreover, such hyperideal polyhedron $P_{hi}$ is unique up to
hyperbolic isometries of $\mathbb{B}^{3}$.
\end{lemma}

As a consequence, Lemma \ref{Lem4} implies there exists an injection
$$
\Psi:Iso^{+}(\mathbb{B}^{3})\times U\rightarrow Z_{P},
$$
where $U$ is the relatively open convex set of $[\pi/2,\pi]^{|E|}$
defined by the constraint conditions $(i)$ and $(ii)$. An elementary
computation shows that this map is a diffeomorphism. Denoting
$\Theta=(\theta_{1},\theta_{2},\cdots,\theta_{E})$, the injection
tells us that there exist $(m_1,\Theta_1)\in
Iso^{+}(\mathbb{B}^{3})\times U$ such that the pushing map:
 \[
 \Psi_{\ast}:T_{m_1}Iso^{+}(\mathbb{B}^{3})\times T_{\Theta_1}U \rightarrow T_{z_1}Z_{P}
 \]
 is a linear isomorphism, where $z_1=\Psi(m_1,\Theta_1)$.

\bigskip
The tangent space $T_{\Theta_1}U$ is expanded by vectors
$\frac{\partial}{\partial\theta_{1}},\frac{\partial}{\partial\theta_{2}},\cdots,\frac{\partial}{\partial\theta_{|\mathcal E|}}$. Note that
\[
dim Z_P=dim T_{\Theta_1}U + dim T_{m_1}Iso^{+}(\mathbb{B}^{3})=|\mathcal E|+6 .
\]
This gives us a geometric insight into the tangent space of $Z_P$.
Then we provide an alternative approach to Theorem \ref{Thm2}.
In order to prove the result, we need the following theorem
concerning the Teichm\"{u}ller theory of circle patterns \cite{H-L}.
\begin{lemma}\label{Lem5}
Suppose that a weight function $\Theta: E\rightarrow [\pi/2,\pi]$
satisfies the conditions $(i)$ and $(ii)$. For any
$$
\big[\tau\big]=\big([\tau_{{1}}],[\tau_{{1}}],\cdots,[\tau_{{n}}]\big)
\in {\mathcal T}_{G^{\ast}(P)},
$$
there exists a unique normalized circle pattern $\mathcal P(\Theta, [\tau])$
with contact graph $G^{\ast}(P)$ and with external dihedral angle
$\Theta(e), \:e\in E$. Moreover, the interstice corresponding to
$I_i$ is endowed with the given complex structure $[\tau_{i}], 1\leq
i\leq n$.
\end{lemma}
Lemma \ref{Lem5} implies that we can define, for each $\Theta \in
U$, a mapping $f_{\Theta}:\mathcal {T}_{G^{\ast}(P)}\rightarrow
Z_{oc}$ via associating every $[\tau]\in \mathcal{T}_{G^{\ast}(P)}$
with the unique normalize circle pattern which realizes the complex
structure$[\tau]$. Denoting $\Theta_0=(\pi,\pi,\cdots,\pi)$ and
$\Theta_s=s\Theta_1+(1-s)\Theta_0, \: s\in [0,1]$, then
$f_{\Theta_s}$ is a homotopy from $f_{\Theta_0}$ to
$f_{\Theta_1}$. By using an argument similar to Lemma \ref{Lem2}, it
follows that there exists a compact subset $\Lambda \subset \mathcal
{T}_{G^{\ast}(P)}$ such that $f_{\Theta_s}(\partial \Lambda)\cap
Z_{P}=\emptyset, \: s\in [0,1]$.

\bigskip
Let's compute the intersection number $I(f_{\Theta_{0}},
\Lambda,Z_P)$. In order to use Lemma \ref{Lem5}, we choose
$\Theta_1$ such that $\Theta_1 \in U\cap[\frac{\pi}{2},\pi]^{|E|}$.
In view of that $\Psi_{\ast}$ is a linear isomorphism, it's
easy to see that:
\begin{proposition}\label{Prop3}
$f_{\Theta_1}\mbox{\Large$\pitchfork$}Z_{P}$.
\end{proposition}
To some extent, the above proposition could be considered as a
substitution of Transversality theorem (Lemma \ref{Lem6}). In fact,
by using an almost repeated procedure as in Section $4$,  we acquire that:
\begin{corollary}
Let $P$ be the given polyhedron. Then
$$
I (f_{K_0}, \Lambda,Z_P)=I (f_{\Theta_0}, \Lambda,Z_P)=I (f_{\Theta_1}, \Lambda,Z_P)=1.
$$
\end{corollary}

Consequently, we are able to obtain an alternative proof of
Theorem \ref{Thm2}, which implies the existence part of Theorem
\ref{main} and Theorem \ref{main1}.

In addition, the procedure of the above proof could be reversed.
That is, since Transversality theorem (Lemma \ref{Lem6}) has been
proved by Schramm independently, combining it with the above
discussion, we could derive the existence part of Lemma
\ref{Lem4}.

\bigskip
\section{Appendix: Fixed point index and Rigidity Lemma}

Recall that $K$ is a given strictly convex body. In this section, it remains to
prove the following lemma concerning the rigidity of $K$-circle
packings.

\begin{lemma}\textbf{(Rigidity Lemma)}\label{Lem7}
Suppose $\mathcal {P}_K,\mathcal {P}'_K$ are two $K$-circle packings
in $\hat{ \mathbb C}$ with the same contact polyhedral graph
$G$. Moreover, assume that they are normalize with the same
mark $\{\mathscr{O},p_1,p_2,p_3\}$.

Denoting by $I\:(resp. \:I')$ the union set of interstices of
$\mathcal {P}_K\:(resp. \:\mathcal {P}'_K)$, if there exists a
conformal mapping $f:I\to I'$, then we
have $\mathcal {P}_K=\mathcal {P}'_K$.
\end{lemma}
\bigskip
In order to prove this lemma, the fixed points index
method is needed. Let us recall its definition.
Please refer to \cite{H-S,Steph} for more details.

Let $\gamma$ be a Jordan curve in the complex plane $\mathbb{C}$.
Suppose $f: \gamma \rightarrow \mathbb{C}$ be a continuous map
without fixed points. The $index(f)$ is defined to be the winding
number of $g\circ \gamma$ around the point $0$, where $g(z)=f(z)-z$
and $\gamma$ is parameterized in accordance with its orientation. In \cite{H-S}, He-Schramm established the following result.

\begin{lemma}\textbf{(Index Lemma)}\label{Lem8}
Let $J$,$J'$ be Jordan curves in $\mathbb C$, positively oriented
with respect to the Jordan domains that they bound; and let $f:J\to
J'$ be an orientation preserving homeomorphism with no fixed points.
Then

(a).If $J$ is contained in the closure of the Jordan domain
determined by $J'$, or $J'$ is contained in the closure of the
Jordan domain determined by $J$, then $index(f)=1$.

(b).If the intersection of $J$ and $J'$ contains at most 2 points, then $index(f)\geq 0$.
\end{lemma}

As an immediate result, we have:

\begin{proposition}\label{Prop4}
If $J,\: J'$ are $K$-circles, then $index(f)\geq0$.
\end{proposition}

Assume that $f:A\to\mathbb{C}$ is continuous, where $A\subset \mathbb
C$. Given an isolated fixed point $z \in int(A)$ (the
interior of $A$) of $f$, then there exists a closed disk $D$
that contains $z$ in its interior, but does not contain any other
fixed point of $f$. The index of $f$ at $z$, denoted by $index(f,z)$, is defined as the restriction of $f$ to $\partial D$, where $\partial D$ is positively oriented with respect to $D$.

 In order to prove Lemma \ref{Lem7}, we shall use the following
result as well. Please refer to \cite{H-S}.

\begin{lemma}\label{Lem9}\textbf{(Poincar\'{e}-Hopf)}
Let $A \subset \mathbb C$ be a compact set whose boundary consists
of finitely many disjoint Jordan curves. Assume that its boundary
components is positively oriented with respect to $A$. Suppose that
$f:A\to \mathbb C$ is continuous, has only isolated fixed points and
has no fixed points on the boundary of $A$. Then the index of the
restriction of $f$ to $\partial A$ is equal to the sum of the
indices of $f$ at all its fixed points.
\end{lemma}

 Let us give the proof of the Rigidity Lemma.
\begin{proof}[\textbf{Proof of Lemma \ref{Lem7}}]
Suppose $\mathscr{O}=\{v_0,e_1,e_2,e_3\}$ is the combinatorial frame associated to $G$. There exist four special $K$-disks $D_0,D_1,D_2,D_3$ of
the packing $\mathcal P_K$ corresponding to $\mathscr{O}$. More
precisely, $D_0=D(v_0)$, $D_i=D(v_i)$ and $e_i=[v_0, v_i]$ for
$i=1,2,3$. Similarly, we have the corresponding disks
$D'_0,D'_1,D'_2,D'_3$ for the packing $\mathcal P'_K$.

Evidently, we have $D_0=D'_0$. Furthermore, we claim that:
\[
D_1=D'_1,\ D_2=D'_2,\ D_3=D'_3.
\]
Otherwise, without loss of generality, we suppose that
 \[
 D_1\subsetneqq D'_1, \ D_2\subsetneqq D'_2, \ D_3\subset D'_3( D_3 \supset D'_3),
 \]
 or
 \[
  D_1\subsetneqq D'_1,\  D_2\subset D'_2(D_2\supset D'_2),\ D_3=D'_3.
  \]

 By means of the fixed point index method, we shall show that both cases lead to contradiction.

  In the first case, according to the normalization assumption, we know that $f$ has fixed points on the boundary $\partial I$ of some $I$.
By post-compositing $f$ with proper linear mapping $az+b$, where $a,b \in \mathbb
C$, we may assume the resulting map $f_{a,b}$ has no fixed points in
$\partial I$, and $f_{a,b}(D_i)\subsetneqq int(D'_i)$ (the interior of $D'_i$) for $i=0,1,2$. Recalling Lemma \ref{Lem8} and Proposition
\ref{Prop4}, a simple computation shows that the index of the
restriction of $f_{a,b}$ to $\partial I$ is less than $-1$. However, because
$f_{a,b}$ is conformal in $I$, by Lemma
\ref{Lem9}, that's impossible.

 For the latter case, due to Koebe's Uniformization \cite{H-S}, we assume that $D_0= D'_0, D_3=D'_3$ are all standard disks.
That means the boundary of these disks are all "real" circles, by
reflection principle, we transform this case to the former one.

Thus we deduce that $D_1=D'_1,\ D_2=D'_2,\ D_3=D'_3$.
Similarly, we assume they are all standard disks.
 With repeated applications of the reflection principle, we shall obtain an extended conformal mapping between the new interstice sets. For ease of notations, we still denote them by $f$, $I$ and $I'$ respectively.

  It's not hard to see that there exists a  circle $C_{\gamma}\subset \bar{I}\cap\bar{I'}$ which is invariant under $f$. Namely, $f(C_{\gamma})=C_{\gamma}$. Moreover, we could assume that $C_{\gamma}$ contains fixed points $u_1,u_2,p_3$, where $u_1,u_2$ are obtained from $p_1,p_2$ by reflections. Denote by $\mathcal{I}\subset C_{\gamma}$ the set of the fixed points $z\in C_{\gamma}$. Note  that each connected component of $\mathcal{I}$ is either a single point, or a closed arch. We claim $\mathcal{I}$ contains at least one closed arch.

Otherwise, without loss of generality, assume that $\mathcal{I}$ consists of  $u_1,u_2,p_3$.  Using the M\"{o}bius transformation which maps $u_1,u_2,p_3$ into $0,1,\infty$, we identify $C_{\gamma}$ with the extended real line $\bar{\mathbb {R}}=\mathbb {R} \cup\{\infty\}$.  Then the restriction of $f$ to $\mathbb {R}$, denoted by $f(x)$, is a real continuous function with fixed points $0,1$.
Let $g(x)=f(x)-x$. We could choose $\varepsilon>0$ sufficiently small such that one of the following two cases occurs:
\begin{itemize}
\item{}$g(-\varepsilon)g(\varepsilon)<0$
\item{}$g(-\varepsilon)g(\varepsilon)>0$.
\end{itemize}

 Perturbing $g(x)$ to suitable $g_{\mu}(x)=g(x)+\mu$, where $0\neq\mu\in \mathbb R$, we shall find either a zero point of $g_\mu$ in $(-\varepsilon,\varepsilon)$ for the first case, or two in the same interval under the second case.  Let us treat another fixed point $1$ of $f(x)$ similarly. Then there exist at least one zero point of $g_\mu$ in $(1-\varepsilon,1+\varepsilon)$ accordingly. Thus, to summarise, we infer that the map $f_\mu(z)=f(z)+\mu$ has at least two fixed points $z_1,z_2$ in $I$. In view of Lemma \ref{Lem9}, a perturbation method implies that this would lead to contradiction.

  Hence we have shown the claim that $\mathcal I\subset C_\gamma\subset I$  contains at least one closed arch. That means the analytic function $g(z)=f(z)-z$ has non-isolated zero point. Therefore, $f(z)=z$. Eventually, we complete the proof of the lemma.
\end{proof}

\end{document}